\documentclass[oneside]{amsart}
\usepackage[a4paper, top=4cm, bottom=4cm]{geometry}
\usepackage{graphicx}
\usepackage{amsmath, amsthm, amssymb, wasysym, verbatim, bbm, color, graphics, geometry, hyperref}
\usepackage{enumitem}
\usepackage{float}
\usepackage{mathtools}
\usepackage{multirow}
\usepackage{siunitx}
\usepackage{stfloats}
\usepackage{thmtools,thm-restate}
\usepackage{threeparttable}
\usepackage{verbatim}
\usepackage{wrapfig}
\usepackage{url}
\usepackage{diagbox}
\usepackage{indentfirst}
\usepackage{listings}
\usepackage{xcolor} 
\usepackage{tikz-cd}
\usepackage{tikz}
\usetikzlibrary{arrows.meta,positioning}
\usepackage{caption}

\theoremstyle{plain}
\newtheorem{theorem}{Theorem}[section]

\newtheorem{prop}[theorem]{Proposition}
\newtheorem*{theorem*}{Theorem}

\theoremstyle{definition}
\newtheorem{definition}[theorem]{Definition}
\theoremstyle{remark}
\newtheorem{rem}[theorem]{Remark}
 \hypersetup{hidelinks}

\begin{document}
\title[A Murphy's law for supports of forgetful maps]{A Murphy's law for supports of forgetful maps between moduli spaces of stable pointed curves}
\author{Hao Zhang}
\address{Peking University}
\email{hao\_zhang@stu.pku.edu.cn}
\date{}
\begin{abstract}
    We study in this note the decomposition theorem for forgetful maps between moduli spaces of genus $g$ stable pointed curves $\overline{\mathcal{M}}_{g,m}\to \overline{\mathcal{M}}_{g,n}$. We prove that as $m$ goes to infinity, every stratum closure in $\overline{\mathcal{M}}_{g,n}$ appears as the support of a direct summand in the decomposition theorem. The proof uses Hassett's moduli spaces of stable weighted pointed curves and a wall-crossing argument. 
\end{abstract}
\maketitle
\section{Introduction}
We work over the field of complex numbers. 

 The decomposition theorem of Beilinson--Bernstein--Deligne--Gabber in \cite{BBD1982} provides a powerful tool for understanding the topology of proper maps between algebraic varieties. Specifically, let $f: X \to S$ be a proper map of complex algebraic varieties. The pushforward of the intersection complex can be decomposed into its shifted perverse cohomology,
\begin{equation*}
    Rf_* \operatorname{IC}_X \simeq \bigoplus_{i \in \mathbb{Z}} {}^p\mathcal{H}^i(Rf_* \operatorname{IC}_X)[-i] \; in  \; D^b_c(S).
\end{equation*}
Furthermore, the perverse sheaves ${}^p\mathcal{H}^i(Rf_* \operatorname{IC}_X)$ are semisimple, i.e., there is a finite stratification into nonsingular subvarieties $S = \bigsqcup_{\beta} S_{\beta}$ and a decomposition into a direct sum of intersection complexes of semisimple local systems,
\begin{equation*}
    {}^p\mathcal{H}^i(Rf_* \operatorname{IC}_X) \simeq \bigoplus_{\beta} \operatorname{IC}_{\overline{S_{\beta}}}(\mathcal{L}_{\beta}).
\end{equation*}

Via the decomposition theorem, one of the basic invariants we can associate to the proper map $f$ is the set of supports occurring in $Rf_*\mathrm{IC}_X$. We are particularly interested in the forgetful maps between moduli spaces of stable pointed curves. In genus $2$, Petersen proved the result for the moduli spaces $\mathcal{M}^{\mathrm{ct}}_{2,m}$ of stable $m$-pointed curves of compact type in \cite[Theorem 2.1]{PE}. The supports which appear in the decomposition theorem for the forgetful map $\mathcal{M}^{\mathrm{ct}}_{2,m}\to \mathcal{M}_2^{\mathrm{ct}}$ are exactly $\mathcal{M}_2^{\mathrm{ct}}$ and $\mathrm{Sym}^2(\mathcal{M}_{1,1})$. 

In this note, we extend the study of the supports in the decomposition theorem to the full Deligne--Mumford moduli spaces $\overline{\mathcal{M}}_{g,n}$ of stable pointed curves. We analyze the decomposition theorem for forgetful maps $\overline{\mathcal{M}}_{g,m}\to \overline{\mathcal{M}}_{g,n}$, and find that after forgetting sufficiently many points, all possible supports will appear.
\begin{theorem}[Main theorem]\label{THM:M}

Let $g,n,m$ be non-negative integers such that \mbox{$2g-2+n>0$} and \mbox{$m>n$}. Let $f_{g,n,m-n}:\overline{\mathcal{M}}_{g,m}\to \overline{\mathcal{M}}_{g,n}$ be the map which forgets the last $m-n$ marked points. Then, for every fixed integer $n$, when $m$ is sufficiently large, every stratum closure appears as the support of a direct summand in the decomposition theorem for ${Rf_{g,n,m-n}}_*\mathbb{Q}_{\overline{\mathcal{M}}_{g,{m}}}[3g-3+m]$.
\end{theorem}
An effective bound for $m$ in terms of $g$ and $n$ is computed in Remark \ref{bd}.

The proof uses the moduli spaces of stable weighted pointed curves introduced by Hassett in \cite{Ha} and we generalize the problem to these more general moduli spaces. We reduce all the computation to the case of the minimal weighted moduli space $\overline{\mathcal{M}}_{g,\epsilon^n}$ (for a detailed definition, see Definition \ref{def}) through the wall-and-chamber structure on the moduli spaces of stable weighted pointed curves. The advantage of $\overline{\mathcal{M}}_{g,\epsilon^n}$ is that all rational tails on a stable curve are contracted and as a result, the cohomology groups of the fibers of forgetful maps become computable. We determine whether a support occurs by explicitly computing the cohomology of the fibers. The paper is organized in the following steps.

\textbf{Step 1:} We review the moduli spaces of stable weighted pointed curves and the wall-and-chamber structure introduced by Hassett. We compute the formula for the decomposition theorem for the reduction maps between differently weighted moduli spaces arising from the wall-crossings, which allows us to recover results for $\overline{\mathcal{M}}_{g,n}$ from the results for $\overline{\mathcal{M}}_{g,\epsilon^n}$.

\textbf{Step 2:} In the second step, we calculate the explicit formula for the decomposition theorem for the one-time forgetful maps, each of which forgets only the last marked point. The one-time forgetful maps $\overline{\mathcal{M}}_{g,n+1}\to \overline{\mathcal{M}}_{g,n}$ and $\overline{\mathcal{M}}_{g,\mathcal{A}\cup\epsilon}\to \overline{\mathcal{M}}_{g,\mathcal{A}}$ (for the notation, see Definition \ref{Ha}) can be viewed as the morphism of the universal curve family. Using this description, we determine the supports occurring in the decomposition theorem.

\textbf{Step 3:} For general forgetful maps, we first restrict the maps to the moduli spaces of curves of compact type. We repeatedly apply the formula for the one-time forgetful maps to the newly appearing direct summands in the decomposition theorem for one-time forgetful maps and describe how the supports change after being pushed forward by one-time forgetful maps step by step. As a special case in Remark \ref{g2}, in genus $2$, we recover the result about supports in the decomposition theorem for $\mathcal{M}^{\mathrm{ct}}_{2,m}\to \mathcal{M}_2^{\mathrm{ct}}$ originally proved by Petersen.

\textbf{Step 4:} Next, we solve the support problem on the divisor which is the complement of the moduli space of curves of compact type by computing the cohomology of general fibers over this divisor. We compute the cohomology of the fibers over the divisor when the forgetful map is $\overline{\mathcal{M}}_{g,\epsilon^m}\to\overline{\mathcal{M}}_{g,\epsilon^n}$ and prove that the divisor $D:=\overline{\mathcal{M}}_{g,\epsilon^n}\backslash\mathcal{M}^{\mathrm{ct}}_{g,\epsilon^n}$ occurs as the support when $m-n$ is sufficiently large. Moreover, we recover the same result for arbitrary weight data because of the wall-crossing formula. The rule governing how the supports contained in the divisor change after being pushed forward by one-time forgetful maps is related to the case of moduli spaces of genus $g-1$ stable weighted pointed curves via the finite morphism from the moduli space of genus $g-1$ stable curves to the divisor. By an inductive argument, we prove the theorem in the generalized case for $\overline{\mathcal{M}}_{g,\mathcal{A}\cup\epsilon^{m-n}}\to\overline{\mathcal{M}}_{g,\mathcal{A}}$ and the method can also be applied to prove the theorem for $\overline{\mathcal{M}}_{g,m}\to\overline{\mathcal{M}}_{g,n}$.
\subsection*{Acknowledgements}
    I would like to thank Qizheng Yin, who guided me in completing this work and supported my study of algebraic geometry. I would also like to thank Terry Dekun Song for many helpful discussions and insights.
\section{Moduli spaces of stable weighted pointed curves and a wall-crossing formula}
In this section, we review the moduli spaces of stable weighted pointed curves and the wall-and-chamber structure on the domain of weight data introduced by Hassett. We provide a wall-crossing formula for the decomposition theorem. 
\subsection{Moduli spaces of stable weighted pointed curves}\label{sec:MS}
\begin{definition}\label{def}

    Given the genus $g$ and weight data $\mathcal{A}:=(a_1,a_2,\ldots,a_n)\in \mathbb{Q}^n$ such that
    \begin{equation}\label{var}
        0< a_i\leq1 ,\quad 2g-2+a_1+a_2+\ldots+a_n>0,
    \end{equation}
    a family of nodal curves of genus $g$ with $n$ marked points $\pi:(C,s_1,s_2,\ldots,s_n)\to B$ is said to be stable of type $(g,\mathcal{A})$  if
    \begin{enumerate}
        \item [(i)] The sections $s_i$ lie in the smooth locus of $\pi$ and for any subset $\{s_{i_1},\ldots,s_{i_r}\}\subset\{s_1,\ldots,s_n\}$ with nonempty intersection we have 
        \begin{equation*}
            a_{i_1}+\ldots+a_{i_r}\leq1,
        \end{equation*}
        \item [(ii)] $K_\pi+a_1s_1+\ldots+a_ns_n$ is $\pi$-relatively ample.
    \end{enumerate}
    All the weight data in this note satisfy \eqref{var}. We also write the weight data as $\alpha_1^{\beta_1}\ldots\alpha_r^{\beta_r}$ when some weights are repeated. This notation means that the first $\beta_1$ weights are all $\alpha_1$, the next $\beta_2$ weights are $\alpha_2$,$\ldots$, and the last $\beta_r$ weights are~$\alpha_r$. We denote the number of weights in $\mathcal{A}$ by $|\mathcal{A}|$.
\end{definition}

\begin{theorem}[{\cite[Theorem 2.1]{Ha}}]
    Let $(g,\mathcal{A})$ be the data as above. There exists a connected, smooth and proper Deligne--Mumford stack $\overline{\mathcal{M}}_{g,\mathcal{A}}$ representing the moduli problem of stable pointed curves of type $(g,\mathcal{A})$. The corresponding coarse moduli space is projective.
\end{theorem}

 Unlike the usual Deligne--Mumford moduli spaces, Hassett's weighted moduli spaces allow marked points to coincide as long as the sum of their weights is no larger than $1$. 
Here are some typical examples. For $\mathcal{A}=(1,\ldots,1)$, this is just the Deligne--Mumford moduli space of stable pointed curves. For $g=1$ and $\mathcal{A}={(\epsilon,\ldots,\epsilon)}$, the stable curve can either be a smooth genus $1$ curve with n marked points or be several $\mathbb{P}^1$ forming a circle with at least one marked point on each $\mathbb{P}^1$. In particular, the curve may consist of a single $\mathbb{P}^1$ with one node obtained by gluing two points together. For $g=0$ and $\mathcal{A}=(1,1,\epsilon,\ldots,\epsilon)$, the stable curve in this case is a chain of $\mathbb{P}^1$ with the two weight $1$ marked points lying at the two ends of the chain and at least one weight $\epsilon$ marked point on each $\mathbb{P}^1$.

\subsection{Wall-crossing}
In this subsection, we describe the decomposition theorem for the reduction maps arising from changes of weights. This provides a relation between $\overline{{\mathcal{M}}}_{g,\epsilon^n}$ and $\overline{{\mathcal{M}}}_{g,\mathcal{A}}$ for arbitrary weight data $\mathcal{A}$ of $n$ weights. Here $\epsilon^n$ means $n$ weight $\epsilon$ marked points as mentioned in Definition \ref{def} and similarly we write $\overline{{\mathcal{M}}}_{g,n}$ as $\overline{{\mathcal{M}}}_{g,1^n}$.

Notice that when changing the weight data, $\overline{{\mathcal{M}}}_{g,\mathcal{A}}$ remains unchanged unless the weights cross one of the relevant walls from one side to the other. In fact, there are two chamber decompositions for the domain of weight data. The coarse one is given by the set of walls 
\begin{equation*}
    W_c=\{\sum_{i\in I}a_i=1:I\subset\{1,2,\ldots,n\},2<|I|\leq n\},
\end{equation*} while the fine one is given by 
\begin{equation*}
    W_f=\{\sum_{i\in I}a_i=1:I\subset\{1,2,\ldots,n\},2\leq|I|\leq n\}.
\end{equation*} We have the following proposition concerning moduli spaces of weights in the same chamber.
\begin{prop}[{\cite[Proposition 5.1]{Ha}}]\label{Ha}
The moduli space of stable weighted pointed curves $\overline{\mathcal{M}}_{g,\mathcal{A}}$ is constant on each coarse chamber while the universal curve family $\mathcal{C}_{g,\mathcal{A}}$ over $\overline{\mathcal{M}}_{g,\mathcal{A}}$ is constant on each fine chamber.
\end{prop}

For two weight data $\mathcal{A}=(a_1,\ldots,a_n)$ and $\mathcal{B}=(b_1,\ldots,b_n)$ such that $a_i\geq b_i$ for any $i\in\{1,2,\ldots,n\}$, there exists a natural reduction map $f:\overline{{\mathcal{M}}}_{g,\mathcal{A}}\to \overline{{\mathcal{M}}}_{g,\mathcal{B}}$, which is obtained by replacing the weights and then stabilizing the curve. Since moduli spaces on the same coarse chamber are isomorphic, the reduction map can be defined between two moduli spaces which are separated by exactly one wall. Let $\mathcal{A}=(a_1,\ldots,a_n)$ and $\mathcal{B}=(b_1,\ldots,b_n)$ such that there exists exactly one subset $I\subset\{1,2,\ldots,n\}$ with $|I|>2$ satisfying 
\begin{equation}\label{con1}
    \sum_{i\in I}a_i>1\ge\sum_{i\in I}b_i,
\end{equation}
and for every other subset $J$ with $|J|>2$, 
\begin{equation}\label{con2}
    \sum_{i\in J}a_i,\sum_{i\in J}b_i\leq 1\quad\textup{or}\quad\sum_{i\in J}a_i,\sum_{i\in J}b_i>1.
\end{equation}
We compute the decomposition theorem for the reduction maps arising from the simplest wall-crossing.

\begin{prop}[Wall-crossing formula]\label{thm:WC}

    Let $\mathcal{A}$, $\mathcal{B}$ be two weight data of $n$ weights and $I$ be the unique subset of $\{1,2,\ldots,n\}$ of more than $2$ elements such that \eqref{con1} holds, and for every other subset $J$ of more than $2$ elements, condition \eqref{con2} holds. Let $f:\overline{{\mathcal{M}}}_{g,\mathcal{A}}\to \overline{{\mathcal{M}}}_{g,\mathcal{B}}$ be the reduction map. Then 
    \begin{equation}\label{WC}
    Rf_*\mathbb{Q}_{\overline{{\mathcal{M}}}_{g,\mathcal{A}}}\simeq \mathbb{Q}_{\overline{{\mathcal{M}}}_{g,\mathcal{B}}}\oplus{\mathbb{Q}_{\overline{\mathcal{M}}_{g,\sum_{i\in I}b_i}\cup\mathcal{B}_I^c}[-2]}\oplus{\ldots}\oplus{\mathbb{Q}_{\overline{\mathcal{M}}_{g,\sum_{i\in I}b_i}\cup\mathcal{B}_I^c}[4-2|I|]}.
    \end{equation}
\end{prop}
\begin{proof}
    The reduction map $f$ is birational and the only exceptional locus is $\overline{{\mathcal{M}}}_{0,1\cup\mathcal{A}_I}\times\overline{{\mathcal{M}}}_{g,1\cup\mathcal{A}_I^c}$. Here the weight data $1\cup\mathcal{A}_I$ is $(1,a_{i_1},\ldots,a_{i_r})$ where $I=\{i_1,\ldots,i_r\}$. Notice that $\overline{{\mathcal{M}}}_{g,1\cup\mathcal{A}_I^c} $ is isomorphic to $\overline{{\mathcal{M}}}_{g,\Sigma_{i\in I}b_i\cup\mathcal{B}_I^c}$ by the assumption on the weight data and Proposition \ref{Ha}. Therefore, the restricted map 
    \begin{equation*} f:\overline{{\mathcal{M}}}_{0,1\cup\mathcal{A}_I}\times\overline{{\mathcal{M}}}_{g,1\cup\mathcal{A}_I^c}\to \overline{{\mathcal{M}}}_{g,\sum_{i\in I}b_i\cup\mathcal{B}_I^c}
    \end{equation*} 
    is a trivial $\overline{{\mathcal{M}}}_{0,1\cup\mathcal{A}_I}$-fibration. The moduli space $\overline{{\mathcal{M}}}_{0,1\cup\mathcal{A}_I}$ is isomorphic to $\mathbb{P}^{|I|-2}$ and the reduction map $f$ is the blow up along the center $\overline{{\mathcal{M}}}_{g,\sum_{i\in I}b_i\cup\mathcal{B}_I^c}$  (see \cite[Section 6.1]{Ha} and \cite[Remark 4.6]{Ha}, the weight data $1\cup\mathcal{A}_I$ is the same as the weight data $\mathcal{A}_{1,1}[|I|+1]$ in Hassett's notation). Since $f$ is the blow up along ${\overline{\mathcal{M}}_{g,\sum_{i\in I}b_i}\cup\mathcal{B}_I^c}$, the direct summands occurring in the decomposition theorem which are supported on ${\overline{\mathcal{M}}_{g,\sum_{i\in I}b_i}\cup\mathcal{B}_I^c}$ are exactly $\mathbb{Q}_{\overline{\mathcal{M}}_{g,\sum_{i\in I}b_i}\cup\mathcal{B}_I^c}[-2],\ldots,\mathbb{Q}_{\overline{\mathcal{M}}_{g,\sum_{i\in I}b_i}\cup\mathcal{B}_I^c}[4-2|I|]$.
\end{proof}
When we replace $\mathbb{Q}_{\overline{{\mathcal{M}}}_{g,\mathcal{A}}}$ by an intersection complex with the coefficient $\mathcal{L}$ of geometric origin, the possible newly appearing supports except $\overline{{\mathcal{M}}}_{g,\mathcal{B}}$ after wall-crossing are contained in ${\overline{\mathcal{M}}_{g,\sum_{i\in I}b_i}\cup\mathcal{B}_I^c}$ because $f$ is an isomorphism outside $\overline{{\mathcal{M}}}_{0,1\cup\mathcal{A}_I}\times\overline{{\mathcal{M}}}_{g,1\cup\mathcal{A}_I^c}$.

Every wall-crossing can be decomposed into the composition of these simple wall-crossings. We can relate the moduli space $\overline{\mathcal{M}}_{g,\mathcal{A}}$ to $\overline{\mathcal{M}}_{g,\epsilon^n}$ for arbitrary weight data $\mathcal{A}$ in this way when $g\neq 0$.
\section{Decomposition theorem for one-time forgetful maps}

It is difficult to directly compute the decomposition theorem for arbitrary forgetful maps $\overline{{\mathcal{M}}}_{g,\epsilon^m}\to \overline{{\mathcal{M}}}_{g,\epsilon^n}$ or $\overline{{\mathcal{M}}}_{g,1^m}\to \overline{{\mathcal{M}}}_{g,1^n}$. However, the decomposition theorem for the one-time forgetful map which forgets only the last marked point can be determined. If the forgotten marked point has sufficiently small weight $\epsilon$, the one-time forgetful map is just the morphism of the universal curve family. The one-time forgetful map between usual Deligne--Mumford moduli spaces of stable pointed curves is also the morphism of the universal curve family. In these cases, we can apply the Goresky--MacPherson inequality to the problem. Only codimension $0$ or $1$ strata closure can occur in the decomposition theorem. We can determine whether these supports appear by computing the cohomology of the fibers.

    \begin{theorem}[Goresky--MacPherson inequality {\cite[Proposition 1]{NBC}}] \label{thm:gm}
        Assume $X$ is smooth or $\mathrm{IC}_X\simeq\mathbb{Q}_X[\dim X]$. Let $f: X \to S$ be a proper morphism with fibers of pure dimension $d$. Let $Z$ be the support of a simple perverse sheaf occurring in $Rf_* \mathbb{Q}_{X}[\dim X]$. Then we have the following inequality
\[
\operatorname{codim}(Z) \le d.
\]
    \end{theorem}

    The nontrivial local systems supported on the smooth locus $\mathcal{M}_{g,1^n}$ (resp.~$\mathcal{M}_{g,\epsilon^n}$) occurring in the decomposition theorem of $f_{g,n}:\overline{{\mathcal{M}}}_{g,1^{n+1}}\to\overline{{\mathcal{M}}}_{g,1^n}$ (resp.~$f_{g,\epsilon^n}:\overline{{\mathcal{M}}}_{g,\epsilon^{n+1}}\to\overline{{\mathcal{M}}}_{g,\epsilon^n}$) come from the nontrivial monodromy of the first cohomology of the universal curve family. One observation is that these local systems can be extended to the larger open subset $\mathcal{M}_{g,1^n}^{\mathrm{ct}}$ (resp.~$\mathcal{M}_{g,\epsilon^n}^{\mathrm{ct}}$) since a stable curve of compact type comes from a smooth curve by either adding some rational tails or contracting some cycles which are trivial cohomology classes in the first cohomology. Denote the local system on $\mathcal{M}_{g,1^n}^{\mathrm{ct}}$ by $\mathcal{L}_{g,n}$. The local system $\mathcal{L}_{g,n}$ is just the pullback of the local system $\mathcal{L}_{g,n-1}$ on $\mathcal{M}_{g,1^{n-1}}^{\mathrm{ct}}$ by the one-time forgetful map. We denote all these local systems by $\mathcal{L}_g$ if it causes no confusion and the local systems on $\mathcal{M}_{g,\epsilon^n}^{\mathrm{ct}}$ are also denoted by $\mathcal{L}_g$ for the same reason. For a general one-time forgetful map $f_{g,\mathcal{A}}:\overline{{\mathcal{M}}}_{g,\mathcal{A}\cup\epsilon}\to\overline{{\mathcal{M}}}_{g,\mathcal{A}}$ coming from the universal curve family, we denote the local system by $\mathcal{L}_{g,\mathcal{A}}$. Similarly when restricted to the compact type locus, $\mathcal{L}_{g,\mathcal{A\cup\epsilon}}$ is the pullback of $\mathcal{L}_{g,\mathcal{A}}$ along $f_{g,\mathcal{A}}$.

    The strata of $\overline{\mathcal{M}}_{g,\mathcal{A}}$ determined by the topological type of the fiber are in one-to-one correspondence with connected stable dual graphs of genus $g$ and $n$ legs with weight data $\mathcal{A}$. Although the ordinary dual graphs do not record collisions of marked points whose sum of weights is less than $1$, over each fixed dual-graph stratum, including its collision subloci, the universal curve is topologically locally trivial as a family of underlying curves. Hence, all $R^i{f_{g,\mathcal{A}}}_*\mathbb{Q}$  are local systems on that stratum. We denote the stratum corresponding to the stable dual graph $G$ by~$S_G$. The strata can be divided into two types depending on whether they are contained in $\mathcal{M}_{g,\mathcal{A}}^{\mathrm{ct}}$. A stratum is contained in $\mathcal{M}_{g,\mathcal{A}}^{\mathrm{ct}}$ if and only if its corresponding dual graph is a tree with $n$ legs. There is a special graph $G_{\mathcal{A},\mathrm{sing}}$ which consists of a single genus $g-1$ vertex, a self-loop and $n$ legs with weight data $\mathcal{A}$. The stratum $S_{G_{\mathcal{A},\mathrm{sing}}}$ is the locus of $1$-nodal curves.

    The decomposition theorem for the one-time forgetful map is of the following form with the notation defined above.
    \begin{theorem}\label{thm:ot}
        Let $n\geq 1$, and let $\mathcal{A}$ be weight data of $n$ weights with respect to genus $g$ in the interior of a fine chamber. Let $f_{g,n}:\overline{{\mathcal{M}}}_{g,1^{n+1}}\to\overline{{\mathcal{M}}}_{g,1^n}$ and $f_{g,\mathcal{A}}:\overline{{\mathcal{M}}}_{g,\mathcal{A}\cup\epsilon}\to\overline{{\mathcal{M}}}_{g,\mathcal{A}}$ be the forgetful maps. Then 
            \begin{multline}\label{ot1}
            R{f_{g,n}}_*\mathbb{Q}_{\overline{{\mathcal{M}}}_{g,1^{n+1}}}[3g-2+n]\\
            \simeq
            \mathbb{Q}_{\overline{{\mathcal{M}}}_{g,1^{n}}}[3g-2+n]\oplus{\mathbb{Q}_{\overline{\mathcal{M}}_{g,1^{n}}}[3g-4+n]}\oplus{\mathrm{IC}_{\overline{\mathcal{M}}_{g,1^n}}(\mathcal{L}_g)}\bigoplus_{G}\mathrm{IC}_{\overline{S}_G},
            \end{multline}
            \begin{multline}\label{ot2}
            R{f_{g,\mathcal{A}}}_*\mathbb{Q}_{\overline{{\mathcal{M}}}_{g,\mathcal{A}\cup\epsilon}}[3g-2+n]\\
            \simeq
            \mathbb{Q}_{\overline{{\mathcal{M}}}_{g,\mathcal{A}}}[3g-2+n]\oplus{\mathbb{Q}_{\overline{\mathcal{M}}_{g,\mathcal{A}}}[3g-4+n]}\oplus{\mathrm{IC}_{\overline{\mathcal{M}}_{g,\mathcal{A}}}(\mathcal{L}_{g,\mathcal{A}})}\bigoplus_{G}\mathrm{IC}_{\overline{S}_G}.
            \end{multline}
        Here the last direct sum is indexed by all stable trees $G$ with exactly one edge. The strata $S_G$ are precisely the codimension $1$ strata contained in the compact type part.
    \end{theorem}
        \begin{proof}
        The proof relies on the property of the universal curve family. According to the wall-and-chamber structure proved in Proposition \ref{Ha}, we can identify $\overline{\mathcal{M}}_{g,1^n\epsilon}$ as $\overline{\mathcal{M}}_{g,1^{n+1}}$. Therefore, the case of $f_{g,n}:\overline{{\mathcal{M}}}_{g,1^{n+1}}\to\overline{{\mathcal{M}}}_{g,1^n}$ is a special case of $f_{g,\mathcal{A}}:\overline{{\mathcal{M}}}_{g,\mathcal{A}\cup\epsilon}\to\overline{{\mathcal{M}}}_{g,\mathcal{A}}$.
        
        Since the one-time forgetful map $f_{g,\mathcal{A}}:\overline{{\mathcal{M}}}_{g,\mathcal{A}\cup\epsilon}\to\overline{{\mathcal{M}}}_{g,\mathcal{A}}$ can be viewed as the universal curve family over $\overline{{\mathcal{M}}}_{g,\mathcal{A}}$ when $\mathcal{A}$ is in the interior of a fine chamber, the fully supported direct summands are exactly $\mathbb{Q}_{\overline{{\mathcal{M}}}_{g,\mathcal{A}}}[3g-2+n]$, ${\mathbb{Q}_{\overline{\mathcal{M}}_{g,\mathcal{A}}}[3g-4+n]}$ and ${\mathrm{IC}_{\overline{\mathcal{M}}_{g,\mathcal{A}}}(\mathcal{L}_{g,\mathcal{A}})}$ arising from the cohomology of the general fibers. By Theorem \ref{thm:gm}, other possible supports occurring in the decomposition theorem are divisors which are the closures of the codimension $1$ strata. Notice that the support condition for an intersection complex is stricter than that for ordinary perverse sheaves. Let $X$ be any algebraic variety and $\mathrm{IC}_X(\mathcal{L})$ be an intersection complex on $X$. The equality in the support condition 
        \begin{equation}\label{str}
        \dim\mathrm{supp}(\mathcal{H}^j(\mathrm{IC}_X(\mathcal{L})))\leq-j
        \end{equation} holds if and only if $j=-\dim X$ which implies that the support has codimension at least $2$ if $j\neq -\dim X$. We deduce from this that ${\mathrm{IC}_{\overline{\mathcal{M}}_{g,\mathcal{A}}}(\mathcal{L}_{g,\mathcal{A}})}$ only contributes to the first cohomology of the general fiber over a divisor. For any stable tree $G$ with only one edge, the fiber over $S_G$ is a curve with two irreducible components. The second cohomology of the fiber is $2$-dimensional. The fully supported direct summands only contribute a $1$-dimensional cohomology group to the second cohomology of the fiber. Therefore, there is a direct summand supported on $\overline{S}_G$. The second cohomology group counts the irreducible components of the fiber. Since there is at least one marked point, the two irreducible components are distinguished by the marked point. Therefore, the monodromy arising from the second cohomology is trivial. Hence, the direct summand supported on the divisor is the intersection complex $\mathrm{IC}_{\overline{S}_G}$.
        
        Next, we prove that these are all the direct summands that can appear. For any stable tree $G$ with only one edge, the contribution of these direct summands to the cohomology of the fibers over $S_G$ is already $1$-dimensional on the zeroth cohomology and $2$-dimensional on the second cohomology. We deduce from this that any other direct summand supported on $\overline{S}_G$ has the form $\mathrm{IC}_{\overline{S}_G}(\mathcal{L})[1]$. Notice that $R{f_{g,\mathcal{A}}}_*\mathbb{Q}_{\overline{{\mathcal{M}}}_{g,\mathcal{A}\cup\epsilon}}[3g-2+n]$ is self-dual under the Verdier duality. Therefore, $\mathrm{IC}_{\overline{S}_G}(\mathcal{L}^\vee)[-1]$ is also a direct summand. However, this indicates that the third cohomology of the fiber is nontrivial and we find a contradiction. For the graph $G_{\mathcal{A},\mathrm{sing}}$, the fiber over $S_{G_{\mathcal{A},\mathrm{sing}}}$ is an irreducible $1$-nodal curve. Its second cohomology is $1$-dimensional which is already obtained from $\mathbb{Q}_{\overline{\mathcal{M}}_{g,1^n}}[3g-4+n]$. If there is a direct summand supported on $\overline{S}_{G_{\mathcal{A},\mathrm{sing}}}$, we can apply the same argument to deduce the existence of $\mathrm{IC}_{\overline{S}_{G_{\mathcal{A},\mathrm{sing}}}}(\mathcal{L}^\vee)[-1]$ as a direct summand which is also a contradiction. Hence, the direct summands supported on divisors are $\mathrm{IC}_{\overline{S}_G}$ where $G$ ranges over all stable trees with only one edge. 
        \end{proof}
        When $n=0$ and $g=2k\geq2$ is even, the formula is slightly different because of the symmetry of the divisor $\mathrm{Sym}^2(\overline{\mathcal{M}}_{k,1})$. The direct summand supported on $\mathrm{Sym}^2(\overline{\mathcal{M}}_{k,1})$ is $\mathrm{IC}_{\mathrm{Sym}^2(\overline{\mathcal{M}}_{k,1})}(\mathbb{Q}_{\mathrm{sign}})$ where the monodromy of $\mathbb{Q}_{\mathrm{sign}}$ arises from the $S_2$ symmetry. When $n=0$ and $g$ is odd, the two irreducible components have different genus. The formula is the same as that in Theorem \ref{thm:ot}. Hence, we obtain the following formula for $n=0$ and $g\geq2$.

\begin{theorem}
    Let $g\geq2$ be the genus and let $f_{g,0}:\overline{{\mathcal{M}}}_{g,1}\to\overline{{\mathcal{M}}}_{g}$ be the forgetful map. If $g=2k$ is even, then
            \begin{multline}\label{even}
            R{f_{g,0}}_*\mathbb{Q}_{\overline{{\mathcal{M}}}_{g,1}}[3g-2]\\
            \simeq
            \mathbb{Q}_{\overline{{\mathcal{M}}}_{g}}[3g-2]\oplus{\mathbb{Q}_{\overline{\mathcal{M}}_{g}}[3g-4]}\oplus{\mathrm{IC}_{\overline{\mathcal{M}}_{g}}(\mathcal{L}_g)}\bigoplus_{G}\mathrm{IC}_{\overline{S}_G}\oplus{\mathrm{IC}_{\mathrm{Sym}^2(\overline{\mathcal{M}}_{k,1})}(\mathbb{Q}_{\mathrm{sign}})}.
            \end{multline}
        Here the direct sum is indexed by all stable trees $G$ with exactly one edge except the tree with two genus $k$ vertices. The strata $S_G$ are precisely the codimension $1$ strata contained in the compact type part except $\mathrm{Sym}^2(\mathcal{M}_{k,1})$.

        If $g$ is odd, then 
             \begin{equation*}
            R{f_{g,0}}_*\mathbb{Q}_{\overline{{\mathcal{M}}}_{g,1}}[3g-2]\\
            \simeq
            \mathbb{Q}_{\overline{{\mathcal{M}}}_{g}}[3g-2]\oplus{\mathbb{Q}_{\overline{\mathcal{M}}_{g}}[3g-4]}\oplus{\mathrm{IC}_{\overline{\mathcal{M}}_{g}}(\mathcal{L}_g)}\bigoplus_{G}\mathrm{IC}_{\overline{S}_G}.
            \end{equation*}
        Here the direct sum is indexed by all stable trees $G$ with exactly one edge. The strata $S_G$ are precisely the codimension $1$ strata contained in the compact type part.
\end{theorem}
\begin{rem}\label{RM}The condition that the weight data $\mathcal{A}$ is in the interior of a fine chamber is necessary. It guarantees that the forgetful map can be viewed as the morphism of the universal curve family, which is wrong if we take $\mathcal{A}=(1/3,1/3,1/3)$.
    The results for $f_{g,n}:\overline{{\mathcal{M}}}_{g,1^{n+1}}\to\overline{{\mathcal{M}}}_{g,1^n}$ and $f_{g,\epsilon^n}:~\overline{{\mathcal{M}}}_{g,\epsilon^{n+1}}\to\overline{{\mathcal{M}}}_{g,\epsilon^n}$ have the same form. However, after replacing the weight of each point with $\epsilon$, the number of possible stable dual graphs decreases and no rational tails occur in the case of~$f_{g,\epsilon^n}$. The key difference is that for any stable graph $G$ with at least one edge and any $m$, the stratum closure $\overline{S}_G\subset \overline{{\mathcal{M}}}_{g,\epsilon^m}$ is never mapped surjectively to $\overline{{\mathcal{M}}}_{g,\epsilon^n}$ by $f_{g,\epsilon^n,m-n}:~\overline{{\mathcal{M}}}_{g,\epsilon^{m}}\to~\overline{{\mathcal{M}}}_{g,\epsilon^n}$ while this is possible for some stable graphs with one genus $g$ vertex and rational tails for $\overline{\mathcal{M}}_{g,1^m}$.
\end{rem}
\section{Proof of the main theorem}
In this section, we determine the supports occurring in the decomposition theorem for an arbitrary forgetful map $f_{g,\mathcal{A},m-n}:\overline{{\mathcal{M}}}_{g,\mathcal{A}\cup\epsilon^{m-n}}\to\overline{{\mathcal{M}}}_{g,\mathcal{A}}$ when $\mathcal{A}$ is in the interior of a fine chamber by restricting the map to the moduli space of compact type curves and its complement. Here, the sufficiently small weight $\epsilon$ depends on the integer $m$. We require that $(m-n)\epsilon_m$ is sufficiently small such that for any subset $I\subset\{1,2,\ldots,n\}$, if the sum of weights in $I$ is less than~$1$, then $\sum_{i\in I}a_i+(m-n)\epsilon_m<1$. In particular, when considering $\mathcal{A}=\epsilon_m^m$, for every fixed $m$, the weight $\epsilon_m$ should satisfy that $m\epsilon_m<1$. The sufficiently small weight is always assumed to satisfy this property when concerning different $m$ and we denote $\epsilon_m$ by $\epsilon$ if it causes no confusion in this section. The key property of this choice of $\epsilon$ is that each time we forget a weight $\epsilon$ marked point, we obtain a morphism of the universal curve family. Therefore, we can repeatedly apply the formula \eqref{ot2} for one-time forgetful maps on the moduli space of compact type curves while on the complement divisor, we determine the supports by computing the cohomology of the singular fiber. The computation is reduced to the case of $f_{g,\epsilon^n,m-n}:\overline{{\mathcal{M}}}_{g,\epsilon^{m}}\to\overline{{\mathcal{M}}}_{g,\epsilon^n}$ via the wall-crossing formula for arbitrary weight data $\mathcal{A}$ of $n$ weights. We show that we can obtain all the possible supports from two specific high-dimensional supports when $m$ is sufficiently large. The theorem for $f_{g,n,m-n}:\overline{\mathcal{M}}_{g,1^m}\to \overline{\mathcal{M}}_{g,1^n}$ can be recovered by the same method as in the proof for $f_{g,\mathcal{A},m-n}:\overline{{\mathcal{M}}}_{g,\mathcal{A}\cup\epsilon^{m-n}}\to\overline{{\mathcal{M}}}_{g,\mathcal{A}}$.
\subsection{The decomposition theorem for the moduli spaces of compact type curves}
In this subsection, we prove that all possible supports occur in the decomposition theorem for $f_{g,\mathcal{A},m-n}:{\mathcal{M}}^{\mathrm{ct}}_{g,\mathcal{A}\cup\epsilon^{m-n}}\to{\mathcal{M}}^{\mathrm{ct}}_{g,\mathcal{A}}$.
\begin{theorem}\label{CT}
    Let $g\geq0$ be the genus and $\mathcal{A}$ be the weight data of $n$ weights in the interior of a fine chamber. For every fixed integer $n$, when $m$ is sufficiently large, every stratum closure appears as the support in the decomposition theorem for ${Rf_{g,\mathcal{A},m-n}}_*\mathbb{Q}_{{\mathcal{M}}^{\mathrm{ct}}_{g,\mathcal{A}\cup\epsilon^{m-n}}}[3g-3+m]$.
\end{theorem}
\begin{proof}
    Every possible support in the decomposition of ${Rf_{g,\mathcal{A},m-n}}_*\mathbb{Q}_{{\mathcal{M}}^{\mathrm{ct}}_{g,\mathcal{A}\cup\epsilon^{m-n}}}[3g-3+m]$ corresponds to a stable tree of genus $g$ with $n$ legs with weight data $\mathcal{A}$. We prove that we can obtain all these stable trees from the graph $G_{0,\mathcal{A}\cup\epsilon^{m-n}}$ which consists of a single genus $g$ vertex with $m$ legs (with weight data $\mathcal{A}\cup\epsilon^{m-n}$) attached to it. 
    
    We explain what new trees we obtain from the direct summand $\mathrm{IC}_{\overline{S}_G}[c]$ (the closure is taken inside the moduli spaces of the compact type curves) after being pushed forward by the one-time forgetful map, where $G$ is a tree-type graph and does not have a nontrivial automorphism. Here, the automorphism of the graph is a bijection on both the set of vertices and the set of edges which preserves the relation between edges and vertices. Moreover, we require the automorphism to fix legs, i.e., the vertex that the leg is attached to remains unchanged under the automorphism. The latter condition corresponds to the requirement that the automorphism of a stable pointed curve should fix the marked points. In this case, we can apply the formula for the one-time forgetful maps to $R{f_{g,\mathcal{A}}}_*\mathrm{IC}_{\overline{S}_{G}}$ via the following commutative diagram
$$\begin{tikzcd}
   \tilde{\overline{S}_{G}}  \arrow[r,"\pi"] \arrow[d,"\tilde{f}"] &\overline{S}_{G} \arrow[d,"f"] \\ \tilde{\overline{S}_{G'}} \arrow[r,"\pi'"] &\overline{S}_{G'}.
\end{tikzcd}$$
Here, the graph $G'$ is obtained from $G$ by deleting the weight $\epsilon$ leg which is forgotten by the one-time forgetful map and then stabilizing. The map $\pi$ is the normalization map of $\overline{S}_{G}$ and $\pi'$ is similarly defined. The normalization $\tilde{\overline{S}_{G}}$ is a fiber product of moduli spaces of stable weighted pointed curves and the map $\tilde{f}$ is the lift of the one-time forgetful map to the normalization. Since the graph $G$ has no nontrivial automorphism, the normalization map $\pi$ is a finite map and generically one-to-one. Since a finite birational map is small, one obtains \begin{equation*}
    R\pi_*\mathbb{Q}_{\tilde{\overline{S}_{G}}}[\dim \tilde{\overline{S}_{G}}]\simeq \mathrm{IC}_{\overline{S}_{G}}.
\end{equation*}
Therefore, we find that 
\begin{equation*}
    Rf_*\mathrm{IC}_{\overline{S}_{G}}\simeq R{\pi'}_* R{\tilde{f}_*}\mathbb{Q}_{\tilde{\overline{S}_{G}}}[\dim \tilde{\overline{S}_{G}}].
\end{equation*}
The lifted forgetful map $\tilde{f}$ is the one-time forgetful map on the factor where the forgotten marked point lies and the identity on other factors. Hence, we can apply \eqref{ot2} to the lifted forgetful map. Let $v$ be the vertex of genus $\tilde{g}$ such that the forgotten leg is attached to $v$. By the formula for the one-time forgetful map \eqref{ot2}, a new graph is obtained from $G$ by deleting the forgotten leg and then splitting the vertex $v$ into at most two vertices whose sum of genus is exactly $\tilde{g}$. Every original edge incident to $v$ is connected to one of the new vertices and the legs originally attached to $v$ are attached to the new vertices in the new graph. Finally, we need to stabilize the graph. Since the map $\pi'$ is also finite, after applying $R{\pi'}_*$, we obtain $\overline{S}_{\tilde{G}}$ occurring as the support, where $\tilde{G}$ is the new graph. Moreover, if the new graph $\tilde{G}$ also carries no nontrivial automorphism, then $\pi'$ is also generically one-to-one on $\overline{S}_{\tilde{G}}$ and hence, the direct summand supported on $\overline{S}_{\tilde{G}}$ can be chosen as the shifted intersection complex with the trivial coefficient by \eqref{ot2}. Therefore, we can repeat this procedure if the assumption on the symmetry of the graphs still holds.

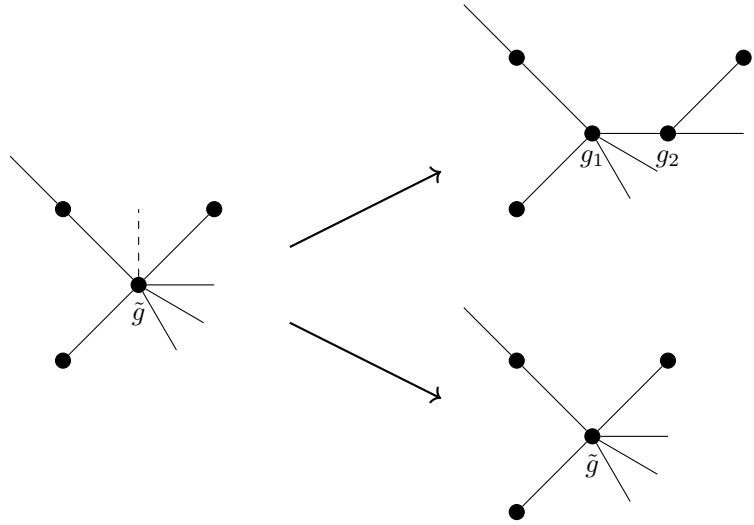
\begin{figure}[H]
\begin{tikzpicture}[
    vertex/.style={circle, draw, fill=black, inner sep=2pt, minimum size=4pt}
  ]
  \begin{scope}
  \node [vertex, label=below:$\tilde{g}$] (A) at (0,0) {};
  \node [vertex, label=below:] (B) at (-1,1) {};
  \node [vertex, label=below:] (C) at (-1,-1) {};
  \node [vertex, label=below:] (D) at (1,1) {};

  \draw (A) -- (B);
  \draw (A) -- (C);
  \draw (A) -- (D);
  \draw (B) -- (-1.7,1.7);
  \draw (A) -- (0.5,-0.86);
  \draw (A) -- (0.86,-0.5);
  \draw (A) -- (1,0);
  \draw[dashed] (A) -- (0,1);
  \end{scope}
  \draw[->, thick] (2,0.5) -- (4,1.5);
  \draw[->, thick] (2,-0.5) -- (4,-1.5);
  \begin{scope}[xshift=6cm, yshift=2cm]
  \node [vertex, label=below:$g_1$] (A) at (0,0) {};
  \node [vertex, label=below:$g_2$] (E) at (1,0) {};
  \node [vertex, label=below:] (B) at (-1,1) {};
  \node [vertex, label=below:] (C) at (-1,-1) {};
  \node [vertex, label=below:] (D) at (2,1) {};

  \draw (A) -- (B);
  \draw (A) -- (C);
  \draw (E) -- (D);
  \draw (A) -- (E);
  \draw (B) -- (-1.7,1.7);
  \draw (A) -- (0.5,-0.86);
  \draw (A) -- (0.86,-0.5);
  \draw (E) -- (2,0);
  \end{scope}
  \begin{scope}[xshift=6cm, yshift=-2cm]
  \node [vertex, label=below:$\tilde{g}$] (A) at (0,0) {};
  \node [vertex, label=below:] (B) at (-1,1) {};
  \node [vertex, label=below:] (C) at (-1,-1) {};
  \node [vertex, label=below:] (D) at (1,1) {};

  \draw (A) -- (B);
  \draw (A) -- (C);
  \draw (A) -- (D);
  \draw (B) -- (-1.7,1.7);
  \draw (A) -- (0.5,-0.86);
  \draw (A) -- (0.86,-0.5);
  \draw (A) -- (1,0);
  \end{scope}
\end{tikzpicture}
\caption{The black dots stand for vertices and a line stands for an edge if it connects two vertices or for a leg if it starts from a vertex and has no end. The dashed line corresponds to the marked point forgotten by the forgetful map. The figure shows how the graph changes under the pushforward of the one-time forgetful map. Here $g_1+g_2=\tilde{g}$.}
  \end{figure}

Consequently, for any stable tree $G$, we first add an extra weight $\epsilon$ leg to each vertex to obtain a new graph $\tilde{G}$. This breaks the symmetry of the graph. Then, we can choose an edge in $\tilde{G}$, glue the two vertices connected by the edge together into a new vertex and add a weight $\epsilon$ leg to it. The graph $\tilde{G}$ can be obtained from this new graph. Denote the number of vertices in $G$ by~$r$. Therefore, the stable tree $\tilde{G}$ of genus $g$ with $n+r$ legs with weight data $\mathcal{A}\cup\epsilon^r$ can be contracted to a graph with a single point and no edges within $3g-3+n$ steps (there are at most $3g-3+n$ edges since the number of edges $e$ is exactly the codimension of the stratum). After forgetting $r=e+1$ added marked points, $\overline{S}_{\tilde{G}}$ is mapped surjectively to $\overline{S}_G$.  Hence, if $m-n\geq6g-5+2n$, every stable tree can be obtained from $G_{0,\mathcal{A}\cup\epsilon^{m-n}}$.
\end{proof}

We are particularly interested in the fully supported direct summands in the decomposition theorem for ${f_{g,\mathcal{A},m-n}}$ since they are the only direct summands which can contribute to the stratum $S_{\mathcal{A},G_{\mathrm{sing}}}$ after extending to the whole moduli spaces. In the special case when $\mathcal{A}=~\epsilon^n$, an observation is that the direct summand supported on the open stratum $\mathcal{M}_{g,\mathcal{A}}^{\mathrm{ct}}$ in the decomposition theorem for $f_{g,\mathcal{A},m-n}:\overline{{\mathcal{M}}}_{g,\mathcal{A}\cup\epsilon^{m-n}}\to\overline{{\mathcal{M}}}_{g,\mathcal{A}}$ always comes from the pushforward of a fully supported direct summand by the one-time forgetful map. The reason is that for any stable graph $G$ with at least one edge, the stratum closure $\overline{S}_G\subset \overline{{\mathcal{M}}}_{g,\mathcal{A}\cup\epsilon^{m-n}}$ is never mapped surjectively to $\overline{{\mathcal{M}}}_{g,\mathcal{A}}$ by $f_{g,\mathcal{A},m-n}:\mathcal{M}_{g,\mathcal{A}\cup\epsilon^{m-n}}^{\mathrm{ct}}\to\mathcal{M}_{g,\mathcal{A}}^{\mathrm{ct}}$ for any $m$ as mentioned in Remark \ref{RM}. Since the local system $\mathcal{L}_{g,\mathcal{A}\cup\epsilon}$ on $\mathcal{M}_{g,\mathcal{A}\cup\epsilon}^{\mathrm{ct}}$ can be viewed as the pullback of the local system $\mathcal{L}_{g,\mathcal{A}}$ on  $\mathcal{M}_{g,\mathcal{A}}^{\mathrm{ct}}$, we can apply the projection formula to the pushforward of $\mathcal{L}_{g,\mathcal{A}\cup\epsilon}$ and find that
\begin{equation*}
R{f_{g,\mathcal{A}}}_*\mathcal{L}_{g,\mathcal{A}\cup\epsilon}\simeq\mathcal{L}_{g,\mathcal{A}}\oplus\mathcal{L}_{g,\mathcal{A}}^{\otimes2}[-1]\oplus\mathcal{L}_{g,\mathcal{A}}[-2]\oplus{S}\quad \textup{in}\quad D_c^b(\mathcal{M}^{\mathrm{ct}}_{g,\mathcal{A}}),
\end{equation*}
where $S$ is the direct sum of all direct summands that are not fully supported. Therefore, we can determine all the fully supported direct summands which occur in the decomposition theorem for $f_{g,\mathcal{A},m-n}:~\overline{\mathcal{M}}_{g,\mathcal{A}\cup\epsilon^{m-n}}\to\overline{\mathcal{M}}_{g,\mathcal{A}}$ through this recursive structure when the weight data $\mathcal{A}=\epsilon^n$. 

\begin{rem}\label{g2}
    In genus $2$, the only two stable trees with no legs are the graph $G_0$ which consists of only one genus $2$ vertex and the graph $G_1$ which consists of two genus $1$ vertices and an edge connecting them. The corresponding stratum closures in $\mathcal{M}_2^{\mathrm{ct}}$ are exactly $\mathcal{M}_2^{\mathrm{ct}}$ and $\mathrm{Sym}^2(\mathcal{M}_{1,1})$ which are the only possible supports in the decomposition theorem. By Theorem \ref{CT}, we find that both $\mathcal{M}_2^{\mathrm{ct}}$ and $\mathrm{Sym}^2(\mathcal{M}_{1,1})$ appear in the decomposition theorem for $\mathcal{M}^{\mathrm{ct}}_{2,m}\to \mathcal{M}_2^{\mathrm{ct}}$ whenever $m\geq1$ since the number of edges of a genus $2$ stable tree with no legs is at most $1$. This recovers the result about supports in genus $2$ by Petersen in \cite[Theorem 2.1]{PE}.
\end{rem}

\subsection{
  The supports of the decomposition theorem contained in
  \texorpdfstring{$\overline{S}_{\mathcal{A},G_{\mathrm{sing}}}$}{S(A,G-sing)}
} In this subsection, we assume that the genus $g$ is nonzero as every stable curve of genus $0$ is of compact type.

We first study the special case when $\mathcal{A}=(\epsilon^n)$. If a direct summand occurring in the decomposition theorem for $f_{g,\epsilon^n,m-n}$ contributes nontrivially to the cohomology of the fiber over the stratum $S_{G_{\epsilon^n,\mathrm{sing}}}$, its support must be $\overline{\mathcal{M}}_{g,\epsilon^n}$ or $\overline{S}_{G_{\epsilon^n,\mathrm{sing}}}=\overline{\mathcal{M}}_{g,\epsilon^n}\backslash \mathcal{M}_{g,\epsilon^n}^{\mathrm{ct}}$. We can determine whether there is a direct summand supported on $\overline{S}_{G_{\epsilon^n,\mathrm{sing}}}$ by comparing the cohomology of the singular fiber with the contribution of these intersection complexes. Let $C$ be a $1$-nodal curve parametrized by a point in the stratum $S_{G_{\epsilon^n,\mathrm{sing}}}$. Then $C$ is obtained by gluing two points on a smooth curve of genus $g-1$. Let $p$ be the unique node. Let $\tilde{C}$ be a stable curve with $m$ weight~$\epsilon$ marked points such that the curve $\tilde{C}$ becomes $C$ after we forget the last $m-n$ marked points and then stabilize the curve. The forgotten marked points must lie in the smooth locus of $C$ or on some unstable irreducible component which is contracted during the stabilization. Since these marked points are of weight $\epsilon$, the extra irreducible components of $\tilde{C}$ form a chain of $\mathbb{P}^1$. The two $\mathbb{P}^1$ at the ends of the chain are each glued to the genus $g-1$ component in $\tilde{C}$ at a point. If there are no extra irreducible components, the genus $g-1$ component is just glued with itself at two points. 

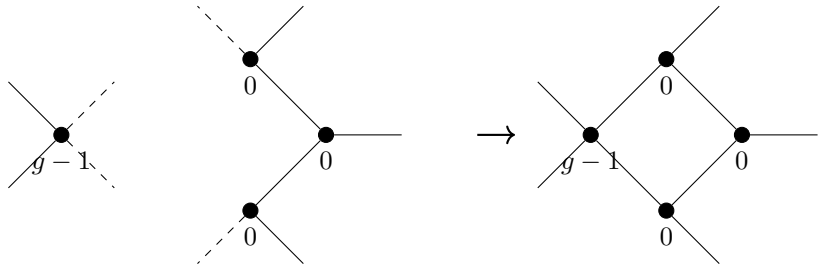
\begin{figure}[H]
\begin{tikzpicture}[
    vertex/.style={circle, draw, fill=black, inner sep=2pt, minimum size=4pt}
  ]
  \begin{scope}
      \node [vertex, label=below:$g-1$] (A) at (0,0) {};
      \draw (A) -- (-0.7,0.7);
      \draw[dashed] (A) -- (0.7,0.7);
      \draw (A) -- (-0.7,-0.7);
      \draw[dashed] (A) -- (0.7,-0.7);
  \end{scope}
  \begin{scope}[xshift=2.5cm]
      \node [vertex, label=below:$0$] (A) at (1,0) {};
      \node [vertex, label=below:$0$] (B) at (0,1) {};
      \node [vertex, label=below:$0$] (C) at (0,-1) {};
      \draw[dashed] (B) -- (-0.7,1.7);
      \draw (B) -- (0.7,1.7);
      \draw (A) -- (2,0);
      \draw (A) -- (B);
      \draw (A) -- (C);
      \draw[dashed] (C) -- (-0.7,-1.7);
      \draw (C) -- (0.7,-1.7);
  \end{scope}
  \draw[->, thick] (5.5,0) -- (6,0);
  \begin{scope}[xshift=8cm]
      \node [vertex, label=below:$0$] (A) at (1,0) {};
      \node [vertex, label=below:$0$] (B) at (0,1) {};
      \node [vertex, label=below:$0$] (C) at (0,-1) {};
      \node [vertex, label=below:$g-1$] (D) at (-1,0) {};
      \draw (B) -- (0.7,1.7);
      \draw (A) -- (2,0);
      \draw (A) -- (B);
      \draw (A) -- (C);
      \draw (C) -- (0.7,-1.7);
      \draw (D) -- (-1.7,0.7);
      \draw (D) -- (-1.7,-0.7);
      \draw (D) -- (C);
      \draw (D) -- (B);
  \end{scope}
  \end{tikzpicture}
\caption{The black dots stand for irreducible components and a line stands for a node if it connects two black dots or for a weight $\epsilon$ marked point if it emanates from a black dot and has no other endpoint. The dashed lines stand for the special points arising from the resolution of the nodal singularities on the genus $g-1$ component. The figure shows how the curve $\tilde{C}$ is decomposed into a smooth genus $g-1$ curve and a chain of $\mathbb{P}^1$.}
  \end{figure}
  
From this observation, we define a filtration on the fiber $F_{m,n}$ over $S_{G_{\epsilon^n,\mathrm{sing}}}$ determined by how many forgotten marked points lie on the genus $g-1$ irreducible component which survives from the stabilization. We denote by $X_l$ the locus of curves with at most $l$ forgotten marked points on this irreducible component. We have
    \begin{equation*}
        \varnothing=X_{-1}\subset X_0\subset\cdots\subset X_{m-n}=F_{m,n}.
    \end{equation*}

    Notice that $X_l$ is closed and $X_l\backslash X_{l-1}$ is the locus of curves with exactly $l$ forgotten marked points on the genus $g-1$ component. $X_l\backslash X_{l-1}$ has a nice description by the moduli space of stable weighted pointed curves of genus $0$. We first choose $l$ points (not necessarily distinct) in the smooth locus $C\backslash \{p\}$ and a chain of $\mathbb{P}^1$ with $m-n-l$ weight $\epsilon$ marked points and two weight~$1$ marked points which are the points at which the chain is glued to the genus $g-1$ component. This chain of $\mathbb{P}^1$ is exactly the curve in $\overline{{\mathcal{M}}}_{0,1^2\epsilon^{m-n-l}}$ mentioned in Section \ref{sec:MS}. In particular, $F_{m,n}$ depends only on $m-n$ and we replace $F_{m,n}$ by $F_{m-n}$ from now on.

    Therefore, $X_l\backslash X_{l-1}=\amalg_{I\subset \{n+1,...,m\},|I|=l}(C\backslash\{p\})^{l}\times\overline{{\mathcal{M}}}_{0,1^2\epsilon^{m-n-l}}$ except for $l=m-n$. In the latter case, 
    \begin{equation*}
        X_l\backslash X_{l-1}=(C\backslash\{p\})^{m-n},
    \end{equation*}
    and there is a spectral sequence by \cite[Lemma 3.8]{Arapura2005}
    \begin{equation*}
        E_1^{p,q}=H^{p+q}_c(X_p\backslash X_{p-1},\mathbb{Q}) \Rightarrow H^{p+q}(F_{m-n},\mathbb{Q}),
    \end{equation*} 
    whose differentials are compatible with the mixed Hodge structure.
\begin{rem}
    The homomorphism in the spectral sequence can be calculated by verifying how the strata are glued together. The stratum $\overline{{\mathcal{M}}}_{0,1^2\epsilon^{m-n-l}}$ is the Losev--Manin space and its cohomology is well studied in \cite{LM} and \cite{BM}. The smooth locus $C\backslash\{p\}$ is topologically isomorphic to a smooth genus $g-1$ curve $C_g$ punctured at two points. This provides a theoretical method to fully determine the cohomology group. In particular, in genus $1$, the intersection complex 
    \begin{equation*}
        \mathrm{IC}_{\overline{\mathcal{M}}_{1,1}}(\mathcal{L}_1^{\otimes k})\simeq (R^0j_*(\mathcal{L}_1^{\otimes k}))[1]
    \end{equation*} 
    where $j:\mathcal{M}_{1,1}\to\overline{\mathcal{M}}_{1,1}$ is the open immersion. Since the stratum $S_{G_{1,\mathrm{sing}}}\subset\overline{\mathcal{M}}_{1,1}$ is a single point, we can compute the Betti numbers of $\overline{{\mathcal{M}}}_{1,\epsilon^m}$ with this spectral sequence and recover the Betti numbers of $\overline{\mathcal{M}}_{1,m}$ by the wall-crossing formula \eqref{WC}. Getzler calculated the Betti numbers of $\overline{\mathcal{M}}_{1,m}$ for small $m$ in \cite[Section 3]{GE}. But he did not provide a closed formula for arbitrary $m$. 
\end{rem}    
    In the following, we use $1_a,1_b$ to distinguish the two weight $1$ marked points. Notice that the product does not affect the gluing. We can glue $(C\backslash\{p\})\times\overline{{\mathcal{M}}}_{0,1^2\epsilon^{m-n-l}}$ and $\overline{{\mathcal{M}}}_{0,1_a1_b\epsilon^{m-n-l+1}}$ and then take the fiber product. The closure of $(C\backslash\{p\})\times\overline{{\mathcal{M}}}_{0,1_a1_b\epsilon^{m-n-l}}$ is $C_g\times\overline{{\mathcal{M}}}_{0,1_a1_b\epsilon^{m-n-l}}$ and it is glued with $\overline{{\mathcal{M}}}_{0,1_a1_b\epsilon^{m-n-l+1}}$ by identifying two $\overline{{\mathcal{M}}}_{0,1_a1_b\epsilon^{m-n-l}}$-sections with 
    \begin{equation*}
        \overline{{\mathcal{M}}}_{0,(1,1_a,\tilde{\epsilon})}\times\overline{{\mathcal{M}}}_{0,(1,1_b,\epsilon^{m-n-l})}\quad\textup{and}\quad \overline{{\mathcal{M}}}_{0,(1,1_b,\tilde{\epsilon})}\times\overline{{\mathcal{M}}}_{0,(1,1_a,\epsilon^{m-n-l})}
    \end{equation*} 
    in $\overline{{\mathcal{M}}}_{0,1^2\epsilon^{m-n-l+1}}$ respectively. Here $\tilde{\epsilon}$ is the newly added weight $\epsilon$ point in $\overline{{\mathcal{M}}}_{0,1^2\epsilon^{m-n-l+1}}$ compared with $\overline{{\mathcal{M}}}_{0,1^2\epsilon^{m-n-l}}$.  

In practice, the computation relies highly on the ring structure of the cohomology ring of the Losev--Manin spaces. However, the spectral sequence rapidly degenerates near the boundary. We can provide a desired bound for the $i$-th cohomology groups if $i$ is large. 

For simplicity, denote $m-n$ by $k$. We start from the top cohomology $H^{2k}(F_k,\mathbb{Q})$. The only relevant differential is 
\begin{equation*}
    d_r:E_{r}^{k-r,k+r-1}\to E_{r}^{k,k}.
\end{equation*}
Notice that $H^{2k-1}_c(X_{k-r}\backslash X_{k-r-1},\mathbb{Q})$ is zero when $r\geq 1$. Thus $d_r$ is the zero homomorphism. Therefore 
\begin{equation*}
    \dim H^{2k}(F_k,\mathbb{Q})=\dim H^{2k}_c((C\backslash\{p\})^{k},\mathbb{Q})=1.
\end{equation*}

For $H^{2k-2}(F_k,\mathbb{Q})$, the relevant differentials are of two types 
\begin{equation}\label{diff}
    \begin{aligned}
        d_r:E_r^{l,2k-2-l}\to E_r^{l+r,2k-1-l-r},\\
        d_r:E_r^{l,2k-3-l}\to E_r^{l+r,2k-2-l-r}.
    \end{aligned}
\end{equation}

\begin{prop}
The two types of differentials in \eqref{diff} are both zero homomorphisms if $l+r\neq k$.
\end{prop}
\begin{proof}
    For the first type of differentials, if $E_r^{l+r,2k-1-l-r}$ is nontrivial, one must have $l+r=k$. Hence they are zero if $l+r\neq k$. For the second type of differentials, the weight argument in the mixed Hodge theory forces most of them to be zero homomorphisms. Indeed, if $l+r<k$, we find that \begin{multline}\label{decom}
         H^{2k-2}_c(X_{l+r}\backslash X_{l+r-1},\mathbb{Q})\\
         =\bigoplus_{I\subset \{1,\ldots,k\},|I|=l+r}H^{2l+2r}_c((C\backslash\{p\})^{l+r},\mathbb{Q})\otimes H^{2k-2l-2r-2}(\overline{{\mathcal{M}}}_{0,1^2\epsilon^{k-l-r}},\mathbb{Q})
    \end{multline}
    by the K\"{u}nneth formula. Therefore, the weight of $E_r^{l+r,2k-2-l-r}$ is $2k-2$. But the highest weight part of $H^{2k-3}_c(X_{l}\backslash X_{l-1},\mathbb{Q})$ is $2k-3$ by the K\"{u}nneth formula and hence the highest weight part of $E_r^{l,2k-3-l}$ is $2k-3$. Since the differentials are compatible with the mixed Hodge structure, the differentials are zero homomorphisms.
    \end{proof}

Notice that the only nontrivial differentials have the target $E_r^{k,k-1}$ or $E_r^{k,k-2}$. Their influence on $H^{2k-2}(F_k,\mathbb{Q})$ is at most $$\dim E_1^{k,k-1}+\dim E_1^{k,k-2}=\binom{k}{1}(2g-1)+\binom{k}{2}(2g-1)^2.$$ Therefore, 
\begin{multline}\label{9}
    \dim H^{2k-2}(F_k,\mathbb{Q})\geq\sum_l\dim H^{2k-2}_c(X_{l}\backslash X_{l-1},\mathbb{Q})-(\binom{k}{1}(2g-1)+\binom{k}{2}(2g-1)^2)\\
    =\binom{k}{k}+...+\binom{k}{2}+\binom{k}{1}+(2g-1)^2\binom{k}{2}-(\binom{k}{1}(2g-1)+\binom{k}{2}(2g-1)^2)=2^k-1-k(2g-1).
\end{multline}
The dimension $\dim H^{2k-2}(F_k,\mathbb{Q})$ also has an upper bound
\begin{equation} \label{ineq2}
    \sum_l\dim H^{2k-2}_c(X_{l}\backslash X_{l-1},\mathbb{Q})=2^k-1+(2g-1)^2\binom{k}{2}\geq \dim H^{2k-2}(F_k,\mathbb{Q})
\end{equation}

With the cohomology groups computed above, we prove the following theorem about the existence of supports in $S_{\mathcal{A},G_{\mathrm{sing}}}$.
\begin{prop}\label{support}
    Let $g\geq1$ be the genus and $\mathcal{A}$ be the weight data of $n$ weights in the interior of a fine chamber. Then $\overline{S}_{\mathcal{A},G_{\mathrm{sing}}}$ occurs as the support in the decomposition theorem for $f_{g,\mathcal{A},k}:\overline{\mathcal{M}}_{g,\mathcal{A}\cup\epsilon^k}\to\overline{\mathcal{M}}_{g,\mathcal{A}}$ when $k$ is sufficiently large. In fact, the shifted intersection complex $\mathrm{IC}_{\overline{S}_{\mathcal{A},G_{\mathrm{sing}}}}[c]$ appears as a direct summand for some integer $c$ when $k$ is sufficiently large.
\end{prop}
\begin{proof}
    Again, we first consider the special case $\mathcal{A}=(\epsilon^n)$. Then, the singular fiber of the forgetful map over the stratum $S_{G_{\epsilon^n,\mathrm{sing}}}$ is $F_k$. By the stricter support condition \eqref{str}, $H^{2k-2}(F_k,\mathbb{Q})$ comes from the intermediate extension of 
    \begin{equation*}
        \mathbb{Q}^{\oplus k}[3g-1+n-k]\bigoplus (\mathcal{L}_{g,\epsilon^n}^{\otimes2})^{\oplus\binom{k}{2}}[3g-1+n-k]
    \end{equation*} 
    on the moduli space of compact type curves or the direct summands supported on $\overline{S}_{G_{\epsilon^n,\mathrm{sing}}}$. Although the intermediate extension is only defined for a local system shifted by $3g-3+n$ on $\mathcal{M}_{g,\epsilon^n}^{\mathrm{ct}}$, we call $\mathrm{IC}_{\overline{\mathcal{M}}_{g,\epsilon^n}}(\mathcal{L})[c-(3g-3+n)]$ the intermediate extension of the local system $\mathcal{L}$ shifted by any integer $c$. The intermediate extension of $\mathcal{L}_{g,\epsilon^n}^{\otimes2}[3g-3+n]$ restricted to $S_{G_{\epsilon^n,\mathrm{sing}}}$ is concentrated in degree $-3g+3-n$ by the stricter support condition \eqref{str}. Suppose that the restriction is a rank $d$ local system on $S_{G_{\epsilon^n,\mathrm{sing}}}$ shifted by $3g-3+n$. Then, if $\overline{S}_{G_{\epsilon^n,\mathrm{{sing}}}}$ does not occur as the support in the decomposition theorem, the cohomology groups of the fiber require that 
    \begin{equation*}
        2^k-1-k(2g-1)\leq k+d\binom{k}{2},
    \end{equation*} 
    by \eqref{9}.
    Notice that $d$ is independent of $k$. When $k$ is large enough, this is impossible for a fixed $d$. Therefore, there exists a direct summand in the decomposition theorem whose support is $\overline{S}_{G_{\epsilon^n,\mathrm{sing}}}$. To prove the second assertion, we analyze the monodromy on the $E_1$ page over $S_{G_{\epsilon^n,\mathrm{sing}}}$. Notice that the filtration can be defined relatively for the fibration over $S_{G_{\epsilon^n,\mathrm{sing}}}$. For each $E_1^{p,q}$ where $p+q=2k-2$, we study the monodromy by the K\"{u}nneth formula \eqref{decom}. In fact, almost all monodromy contributions are trivial. The monodromies on $H^{2k-2l-2}(\overline{{\mathcal{M}}}_{0,1^2\epsilon^{k-l}},\mathbb{Q})$ and $H^2_c(C\backslash\{p\},\mathbb{Q})$ are both trivial. The only nontrivial monodromy comes from $H^{2k-2}({(C\backslash\{p\})}^k,\mathbb{Q})$ over $S_{G_{\epsilon^n,\mathrm{sing}}}$. Even if we wipe out the full influence of the differentials, however, the number of copies of trivial local systems in the variation of Hodge structures of $H^{2k-2}(F_k,\mathbb{Q})$ is at least 
    \begin{equation*}
        \binom{k}{k}+...+\binom{k}{2}+\binom{k}{1}-(\binom{k}{1}(2g-1)+\binom{k}{2}(2g-1)^2)=2^k-1-k(2g-1)-\binom{k}{2}(2g-1)^2.
    \end{equation*} 
    Therefore, when $k$ is sufficiently large such that 
    \begin{equation}\label{ineq}
        2^k-1-k(2g-1)-\binom{k}{2}(2g-1)^2>k+d\binom{k}{2},
    \end{equation} 
    there must be a shifted intersection complex $\mathrm{IC}_{\overline{S}_{\mathcal{A},G_{\mathrm{sing}}}}[c]$ occurring in the decomposition theorem for $f_{g,\mathcal{A},k}$ due to the semisimplicity of the coefficient on $S_{\mathcal{A},G_{\mathrm{sing}}}$ which is guaranteed by the decomposition theorem.

    For the general case, we consider the following commutative diagram
    $$\begin{tikzcd}
   \overline{{\mathcal{M}}}_{g,\mathcal{A}\cup\epsilon^k}  \arrow[r] \arrow[d,"f_{g,\mathcal{A},k}"] &\overline{{\mathcal{M}}}_{g,\epsilon^{n+k}} \arrow[d,"f_{g,\epsilon^n,k}"] \\ \overline{{\mathcal{M}}}_{g,\mathcal{A}} \arrow[r] &\overline{{\mathcal{M}}}_{g,\epsilon^n}.
\end{tikzcd}$$
The horizontal maps are reduction maps induced by changes in the weights. The supports occurring in the decomposition theorem for $f_{g,\mathcal{A},k}$ are the closures of the dual-graph strata. The direct summand in the decomposition theorem for $f_{g,\mathcal{A},k}$ which contributes to the direct summand supported on $\overline{S}_{G_{\epsilon^n,\mathrm{sing}}}$ must be mapped surjectively to the divisor $\overline{S}_{G_{\epsilon^n,\mathrm{sing}}}$ by the reduction map. Therefore, its support must be $\overline{S}_{G_{\mathcal{A},\mathrm{sing}}}$ or $\overline{\mathcal{M}}_{g,\mathcal{A}}$. Notice that for the newly appearing support except the full support in the wall-crossing formula \eqref{WC} for the reduction map that crosses exactly one wall, the codimension is at least $2$. Therefore, a codimension $1$ support cannot be obtained from $\overline{\mathcal{M}}_{g,\mathcal{A}}$ through the reduction map. The direct summand in the decomposition theorem for $f_{g,\mathcal{A},k}$ that contributes to the direct summand supported on $\overline{S}_{G_{\epsilon^n,\mathrm{sing}}}\subset \overline{\mathcal{M}}_{g,\epsilon^n}$ after the wall-crossing is supported on $\overline{S}_{G_{\mathcal{A},\mathrm{sing}}}\subset \overline{\mathcal{M}}_{g,\mathcal{A}}$. Moreover, if the coefficient of the intersection complex supported on $\overline{S}_{G_{\epsilon^n,\mathrm{sing}}}$ is trivial, the corresponding direct summand also has the trivial coefficient. Hence, the proposition is also true in the general case.
\end{proof}
\subsection{
  Main theorem for
  \texorpdfstring{
    $f_{g,\mathcal{A},m-n}:
    \overline{\mathcal{M}}_{g,\mathcal{A}\cup\epsilon^{m-n}}
    \to
    \overline{\mathcal{M}}_{g,\mathcal{A}}$
    }{weighted forgetful maps}
} 
    \begin{theorem}[Main theorem for $f_{g,\mathcal{A},m-n}:\overline{\mathcal{M}}_{g,\mathcal{A}\cup\epsilon^{m-n}}\to\overline{\mathcal{M}}_{g,\mathcal{A}}$]\label{thm:mt}
    Let $g\geq0$ be the genus and $\mathcal{A}$ be the weight data of $n$ weights in the interior of a fine chamber. For every fixed integer $n$, when $m$ is sufficiently large, every stratum closure appears as the support of a direct summand in the decomposition theorem for ${Rf_{g,\mathcal{A},m-n}}_*\mathbb{Q}_{\overline{{\mathcal{M}}}_{g,\mathcal{A}\cup\epsilon^{m-n}}}[3g-3+m]$.
    \end{theorem}
\begin{proof}
    Every possible support in the decomposition of ${Rf_{g,\mathcal{A},m-n}}_*\mathbb{Q}_{\overline{{\mathcal{M}}}_{g,\mathcal{A}\cup\epsilon^{m-n}}}[3g-3+m]$ corresponds to a connected stable dual graph of genus $g$ with $n$ legs with weight data $\mathcal{A}$. We want to obtain all these stable graphs from the graph $G_{0,\mathcal{A}\cup\epsilon^{m-n}}$ which consists of a single genus $g$ vertex with $m$ legs (with weight data $\mathcal{A}\cup\epsilon^{m-n}$) attached to it. We divide the graphs into two types depending on whether they contain a circle. For the stable trees, it is already proved that we can obtain all tree-type graphs when $m$ is sufficiently large in Theorem \ref{CT}. In particular, the theorem is true when $g=0$ since in this case all graphs are tree-type.

    For the graphs with a circle, we consider the direct summand supported on $\overline{S}_{G_{\mathcal{A}\cup\epsilon^{m-n},\mathrm{sing}}}$ whose existence is guaranteed by Proposition \ref{support}.
    We prove that all the graphs with a circle can be obtained from this graph when $m$ is large enough. We reduce the problem to the lower genus case via the following commutative diagram
    $$\begin{tikzcd}
   \overline{{\mathcal{M}}}_{g-1,1^2\cup\mathcal{A}\cup\epsilon^{m-n}}  \arrow[r,"\pi_m"] \arrow[d,"f_{g-1,1^2\cup\mathcal{A},m-n}"] &\overline{S}_{G_{\mathcal{A}\cup\epsilon^{m-n},\mathrm{sing}}} \arrow[d,"f_{g,\mathcal{A},m-n}"] \\ \overline{{\mathcal{M}}}_{g-1,1^2\cup\mathcal{A}} \arrow[r,"\pi_{n}"] &\overline{S}_{G_{\mathcal{A},\mathrm{sing}}}.
   \end{tikzcd}$$
    We can glue the two weight $1$ marked points on a curve in $\overline{{\mathcal{M}}}_{g-1,1^2\cup\mathcal{A}\cup\epsilon^{m-n}}$ and the new curve is a genus $g$ curve in $\overline{S}_{G_{\mathcal{A}\cup\epsilon^{m-n},\mathrm{sing}}}$. This defines the finite map $\pi_n$ in the horizontal direction and $\pi_m$ is similarly defined. The map $f_{g-1,1^2\cup\mathcal{A},m-n}$ is the forgetful map which forgets the last $m-n$ weight $\epsilon$ marked points. The idea is that we can break the circle by lifting to $\overline{{\mathcal{M}}}_{g-1,1^2\cup\mathcal{A}\cup\epsilon^{m-n}}$ and then recover the circle after applying $R\pi_{{n}_*}$. The existence of every graph is converted to the problem on the moduli space of genus $g-1$ stable weighted pointed curves which can be handled by induction, since we have already proved the result for genus $0$.
    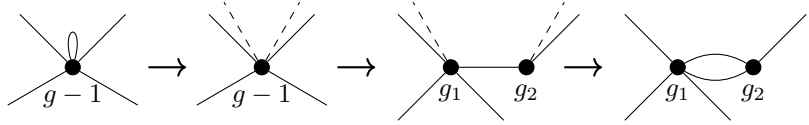
\begin{figure}[H]
\centering
\begin{tikzpicture}[
    vertex/.style={circle, draw, fill=black, inner sep=2pt, minimum size=4pt}
  ]
  \begin{scope}
      \node [vertex, label=below:$g-1$] (A) at (0,0) {};
      \draw (A) to[loop above] (A);
      \draw (A) -- (-0.7,0.7);
      \draw (A) -- (0.7,0.7);
      \draw (A) -- (-0.86,-0.5);
      \draw (A) -- (0.86,-0.5);
  \end{scope}
  \draw[->, thick] (1,0) -- (1.5,0);
  \begin{scope}[xshift=2.5cm]
      \node [vertex, label=below:$g-1$] (A) at (0,0) {};
      \draw (A) -- (-0.7,0.7);
      \draw (A) -- (0.7,0.7);
      \draw (A) -- (-0.86,-0.5);
      \draw (A) -- (0.86,-0.5);
      \draw[dashed] (A) -- (-0.5,0.86);
      \draw[dashed] (A) -- (0.5,0.86);
  \end{scope}
  \draw[->, thick] (3.5,0) -- (4,0);
  \begin{scope}[xshift=5cm]
      \node [vertex, label=below:$g_1$] (A) at (0,0) {};
      \node [vertex, label=below:$g_2$] (B) at (1,0) {};
      \draw (A) -- (B);
      \draw (A) -- (-0.7,0.7);
      \draw (A) -- (-0.7,-0.7);
      \draw (A) -- (0.7,-0.7);
      \draw (B) -- (1.7,0.7);
      \draw[dashed] (A) -- (-0.5,0.86);
      \draw[dashed] (B) -- (1.5,0.86);
  \end{scope}
  \draw[->, thick] (6.5,0) -- (7,0);
  \begin{scope}[xshift=8cm]
      \node [vertex, label=below:$g_1$] (A) at (0,0) {};
      \node [vertex, label=below:$g_2$] (B) at (1,0) {};
      \draw (A) to[bend right=30] (B);
      \draw (A) to[bend left=30] (B);
      \draw (A) -- (-0.7,0.7);
      \draw (A) -- (-0.7,-0.7);
      \draw (A) -- (0.7,-0.7);
     \draw (B) -- (1.7,0.7);
  \end{scope}
  \end{tikzpicture}
\caption{The notation is the same as Figure 1 except for the meaning of the dashed lines. The two dashed lines stand for two special points coming from the resolution of the nodal singularity. The second arrow is the pushforward by the one-time forgetful map in genus $g-1$. The argument made below explains how we can break the self-loop (the first arrow) and then glue the two special points to form a new edge (the third arrow) via the finite map.}
  \end{figure}

    By Proposition \ref{support}, there is a direct summand $\mathrm{IC}_{\overline{S}_{G_{\mathcal{A}\cup\epsilon^{m-n},\mathrm{sing}}}}[c]$ occurring in the decomposition theorem for $f_{g,\mathcal{A}\cup\epsilon^{m-n},k}:\overline{{\mathcal{M}}}_{g,\mathcal{A}\cup\epsilon^{m-n+k}}\to\overline{{\mathcal{M}}}_{g,\mathcal{A}\cup\epsilon^{m-n}}$ when $k$ is large. Therefore, we only need to consider the trivial local system. The general fiber of the finite map $\pi_m$ consists of two disjoint points. There is a dense open subset $U$ of $\overline{S}_{G_{\mathcal{A}\cup\epsilon^{m-n},\mathrm{sing}}}$ such that $\pi_m:\pi_m^{-1}(U)\to U$ is a $2$-fold {\'e}tale covering. Since $\pi_m$ is finite and proper, $R\pi_{m*}$ is perverse t-exact and 
    \begin{equation*}
        R\pi_{m*}\mathbb{Q}_{\overline{{\mathcal{M}}}_{g-1,1^2\cup\mathcal{A}\cup\epsilon^{m-n}}}[3g-4+m]\simeq \mathrm{IC}_{\overline{S}_{G_{\mathcal{A}\cup\epsilon^{m-n},\mathrm{sing}}}}(R^0\pi_{m*}\mathbb{Q}_{{\pi_m}^{-1}(U)}).
    \end{equation*}
    The local system $R^0\pi_{m*}\mathbb{Q}_{{\pi_m}^{-1}(U)}$ is of rank $2$ and can be decomposed into $\mathbb{Q}\oplus\mathbb{Q}_{\mathrm{sign}}$, where the monodromy of $\mathbb{Q}_{\mathrm{sign}}$ comes from swapping the two disjoint points. 
    
     Assume that the theorem is true for genus $g-1$. Then when $m$ is sufficiently large, every connected stable graph for $\overline{\mathcal{M}}_{g-1,1^2\cup\mathcal{A}}$ can be reached. The finite map $\pi_n$ just deletes the two weight $1$ legs and draws an edge between the two vertices with the weight $1$ legs. The last thing we need to verify is that the direct summand indeed comes from $\mathbb{Q}$ instead of $\mathbb{Q}_{\mathrm{sign}}$. Notice that there is an involution $\tau_m$ on $\overline{{\mathcal{M}}}_{g-1,1^2\cup\mathcal{A}\cup\epsilon^{m-n}}$ by swapping two weight $1$ marked points and similarly we can define $\tau_n$ on $\overline{{\mathcal{M}}}_{g-1,1^2\cup\mathcal{A}}$ and they are compatible with the forgetful map. Since $\pi_n\circ\tau_n=\pi_n$, $\tau_n$ induces an endomorphism on 
     \begin{equation*}
         R{f_{g,\mathcal{A},m-n}}_*R{\pi_m}_*\mathbb{Q}_{\overline{{\mathcal{M}}}_{g-1,1^2\cup\mathcal{A}\cup\epsilon^{m-n}}}\simeq R{\pi_n}_* R{f_{g-1,1^2\cup\mathcal{A},m-n}}_*\mathbb{Q}_{\overline{{\mathcal{M}}}_{g-1,1^2\cup\mathcal{A}\cup\epsilon^{m-n}}}.
     \end{equation*}
     Suppose that $\mathcal{L}$ is the direct sum of all summands in the decomposition theorem with a given support and $\rho:\pi_1(B) \to \mathrm{GL}(V)$ is the corresponding representation where $B$ is the underlying space of the local systems. Then $\tau_n$ induces an automorphism on $V$ which is compatible with $\rho$. The local systems corresponding to eigenspaces with eigenvalue $1$ come from $\mathbb{Q}$ while the local systems corresponding to eigenspaces with eigenvalue $-1$ come from $\mathbb{Q}_{\mathrm{sign}}$. If a lifted graph $G$ and its transform $\tau_n G$ are distinct, the direct sum of the summands in the decomposition theorem for $R{f_{g-1,1^2\cup\mathcal{A},m-n}}_*\mathbb{Q}_{\overline{{\mathcal{M}}}_{g-1,1^2\cup\mathcal{A}\cup\epsilon^{m-n}}}$ supported on $\overline{S}_G\cup\overline{S}_{\tau_n G}$ is an induced $S_2$-object. Denote the glued graph by $G_{\mathrm{glue}}$. Then, the local system on $S_{G_{\mathrm{glue}}}$ is the direct sum of two isomorphic parts obtained from $G$ and $\tau_n G$ respectively. At the representation level, the local system corresponding to a representation on $V\oplus{V}$. The automorphism induced by $\tau_n$ is given by 
     \begin{equation*}
         \tau_n(v_1,v_2)=(\tau(v_2),\tau(v_1)),
     \end{equation*}
     where $\tau$ is an automorphism of $V$ satisfying that $\tau^2=\mathrm{id}$. The nonzero subspace consisting of $(v,\tau(v))$ is contained in the eigenspace associated with the eigenvalue $1$. Therefore, there is a direct summand supported on $\overline{S}_{G_{\mathrm{glue}}}$ coming from the trivial local system. The exceptional case that $G$ and $\tau_n G$ are the same can be avoided by the same method that is applied to break the symmetry of a graph in the compact type case. We replace one of the special marked points by a $\mathbb{P}^1$ and move the replaced marked point to this $\mathbb{P}^1$. Moreover, we add a weight $\epsilon$ marked point which will later be forgotten on the $\mathbb{P}^1$. Then on this new graph $\tilde{G}$, the eigenvalue $1$ subspace is nonzero and there is a direct summand belonging to the pushforward of the trivial coefficient supported on $\overline{S}_{\tilde{G}_{\mathrm{glue}}}$, where $\tilde{G}_{\mathrm{glue}}$ is obtained from $\tilde{G}$ by deleting the two special markings and then drawing an edge. After forgetting the newly added marked point, $\overline{S}_{\tilde{G}_{\mathrm{glue}}}$ is mapped surjectively to $\overline{S}_{G_{\mathrm{glue}}}$. Hence, the gluing is valid for any graph $G$. Therefore, all stable graphs of genus $g$ with $n$ legs can be obtained from $G_{\mathcal{A}\cup\epsilon^{m-n},\mathrm{sing}}$ when $m$ is sufficiently large. The base case of the induction is genus $0$ which is already proved. Hence, by induction, both tree-type graphs and graphs with a circle appear for any genus $g$ and any weight data $\mathcal{A}$ when $m$ is large enough and these are exactly all the strata.
\end{proof}
\subsection{
  Main theorem for
  \texorpdfstring{$\overline{\mathcal{M}}_{g,1^m}$}{ordinary pointed moduli spaces}
}

We now prove the theorem for $\overline{\mathcal{M}}_{g,1^m}$ in almost the same way.

\begin{proof}[Proof of Theorem \ref{THM:M}:] The rule governing how a graph changes after being pushed forward by a one-time forgetful map is the same as that in the case of $\overline{\mathcal{M}}_{g,\mathcal{A}}$ by \eqref{ot1}. In particular, the theorem holds when $g=0$. We only need the existence of a direct summand supported on $\overline{S}_{G_{1^m,\mathrm{sing}}}$ to prove the theorem by induction and this is guaranteed by the wall-crossing formula. We apply the same argument as in the proof of Proposition \ref{support}. Consider the following commutative diagram
$$\begin{tikzcd}
   \overline{{\mathcal{M}}}_{g,1^{m+k}}  \arrow[r] \arrow[d,"f_{g,m,k}"] &\overline{{\mathcal{M}}}_{g,\epsilon^{m+k}} \arrow[d,"f_{g,\epsilon^{m},k}"] \\ \overline{{\mathcal{M}}}_{g,1^m} \arrow[r] &\overline{{\mathcal{M}}}_{g,\epsilon^m}.
\end{tikzcd}$$
When $k$ is large, there is a shifted intersection complex $\mathrm{IC}_{\overline{S}_{G_{\epsilon^m,\mathrm{sing}}}}[c]$ in the decomposition theorem for $f_{g,\epsilon^m,k}$. Notice that for the newly appearing support except the full support in the wall-crossing formula \eqref{WC} for the reduction map that crosses exactly one wall, the codimension is at least $2$. Since $\overline{S}_{G_{\epsilon^m,\mathrm{sing}}}$ is a divisor, on the $\overline{\mathcal{M}}_{g,1^m}$ side, there is a direct summand in the decomposition theorem for $f_{g,m,k}$ whose support is $\overline{S}_{G_{m,\mathrm{sing}}}\subset \overline{\mathcal{M}}_{g,1^m}$ and the coefficient should also be trivial. Therefore, we can apply the same argument to deduce that every possible stratum closure appears as the support when $m$ is sufficiently large.
\end{proof}

\begin{rem}\label{bd}
    We provide a effective bound for $m$ by solving the inequalities appearing in the proof. For a tree-type dual graph, by Theorem \ref{CT}, if $m-n\geq6g-5+2n$, its corresponding stratum closure appears. For a graph with a circle, we first need to fix a positive integer $\tilde{m}$ such that every stratum closure in $\overline{S}_{G_{n,sing}}$ can be obtained from the pushforward of the shifted intersection complex $\mathrm{IC}_{\overline{S}_{G_{\tilde{m},\mathrm{sing}}}}[c]$. To apply the induction hypothesis, the integer $\tilde{m}$ should satisfy that every strata closure appear in the decomposition theorem for $f_{g-1,3+n,\tilde{m}-n}:\overline{\mathcal{M}}_{g-1,1^{3+\tilde{m}}}\to\overline{\mathcal{M}}_{g-1,1^{3+n}}$. Next, we need to find a sufficiently large $m$ such that $\mathrm{IC}_{\overline{S}_{G_{\tilde{m},\mathrm{sing}}}}[c]$ occurs as a direct summand in the decomposition theorem for $f_{g,\tilde{m},m-\tilde{m}}$. Denote $m-\tilde{m}$ by $k$. Then, by the proof of Proposition ~\ref{support}, $k$ should satisfy \eqref{ineq} 
    \begin{equation*}
        2^k-1-k(2g-1)-\binom{k}{2}(2g-1)^2>k+d\binom{k}{2},
    \end{equation*}
    and $d$ is controlled by 
    \begin{equation*}
        l+d\binom{l}{2}\leq \dim H^{2l-2}(F_l,\mathbb{Q}),
    \end{equation*}
    for any $l\geq 1$. Take $l=2$ and by \eqref{ineq2}, we find that 
    \begin{equation*}
        d\leq 1+(2g-1)^2.
    \end{equation*}
    Therefore, we can rewrite \eqref{ineq} into a stronger inequality
    \begin{equation*}
        2^k-1-k(2g-1)-\binom{k}{2}(2g-1)^2>k+(1+(2g-1)^2)\binom{k}{2},
    \end{equation*}
    and $k\geq4g+6$ is a solution to this inequality. 
    
    Now we can choose a bound $b(g,n)$ for any $g,n$ inductively such that if $m-n\geq b(g,n)$, every stratum closure appears in the decomposition theorem for $f_{g,n,m-n}$. When $g=0$, the bound $b(g,n)=6g-5+2n$ is enough since every dual graph is a tree in this case. For $g\geq1$, if $b(g-1,n)$ is chosen, we can choose $b(g,n)=4g+6+b(g-1,n+3)$ and the corresponding $\tilde{m}$ is chosen as $\tilde{m}=n+b(g-1,n+3)$. According to the analysis above, this bound works and we compute the closed formula for the bound
     \begin{equation*}
        b(g,n)=2g(g+1)+6g+2n+6g-5=2n+2g^2+14g-5.
    \end{equation*}
\end{rem}

\end{document}